\documentclass[10pt,a4paper]{amsart}
\usepackage{amssymb, amsmath, amsthm,amsfonts}
\usepackage[utf8]{inputenc}
\usepackage{aliascnt}
\usepackage[arrow, matrix, curve, graph]{xy}
\usepackage{longtable}
\usepackage{tikz}
\usetikzlibrary{matrix,arrows,positioning,backgrounds,shapes}
\usepackage{mathtools}

\theoremstyle{plain} 
\newtheorem{lemma}{Lemma}

\newaliascnt{prop}{lemma}
\newtheorem{prop}[prop]{Proposition}

\aliascntresetthe{prop}

\newaliascnt{coro}{lemma}

\aliascntresetthe{coro}

\newaliascnt{theo}{lemma}
\newtheorem{theo}[theo]{Theorem}

\aliascntresetthe{theo}

\newtheorem*{theo*}{Theorem}

\theoremstyle{definition}

\newaliascnt{defn}{lemma}

\aliascntresetthe{defn}

\theoremstyle{remark}

\newaliascnt{rem}{lemma}
\newtheorem{rem}[rem]{Remark}

\aliascntresetthe{rem}

\begin{document}

\title{Subgroup Isomorphism Problem for units of integral group rings}
\author{Leo Margolis}
\date{\today}

\begin{abstract} The Subgroup Isomorphism Problem for Integral Group Rings asks for which finite groups $U$ it is true that if $U$ is isomorphic to a subgroup of $\mathrm{V}(\mathbb{Z}G)$, the group of normalized units of the integral group ring of the finite group $G$, it must be isomorphic to a subgroup of $G$. The smallest groups known not to satisfy this property are the counterexamples to the Isomorphism Problem constructed by M. Hertweck. However the only groups known to satisfy it are cyclic groups of prime power order and elementary-abelian $p$-groups of rank 2. We prove the Subgroup Isomorphism Problem for $C_4 \times C_2$. Moreover we prove that if the Sylow 2-subgroup of $G$ is a dihedral group, any 2-subgroup of $\mathrm{V}(\mathbb{Z}G)$ is isomorphic to a subgroup of $G.$ 
\end{abstract}

\maketitle

\renewcommand{\thefootnote}{}
  \footnote{This research was partly supported by the Sonderforschungsbereich 701 at the University of Bielefeld}

  \vspace{.5cm}
\section{Introduction}

The first to study the unit group of an integral group ring $\mathbb{Z}G$ of a finite group $G$ was G. Higman in his PhD thesis \cite{HigmanThesis}. There, in \textsection 5, he states it "plausible" that any finite subgroup of units $U$ is isomorphic to a group of trivial units, i.e. units of the form $\pm g$ with $g \in G.$ See \cite{Sandling} for details on the thesis. Denote by $\mathrm{V}(\mathbb{Z}G)$ the normalized units of $\mathbb{Z}G$, i.e. the units whose coefficients sum up to 1. In view of the fact that $\pm \mathrm{V}(\mathbb{Z}G)$ are all the units of $\mathbb{Z}G$ it is enough to study normalized units.

After Higman's question remained open for quite a long time, a negative answer to it was given by M. Hertweck's counterexample to the isomorphism problem \cite{HertweckIso}: There are finite groups $G$ and $H$ such that $\mathbb{Z}G$ is isomorphic to $\mathbb{Z}H$, but $G$ and $H$ are not isomorphic. In particular, $\mathrm{V}(\mathbb{Z}G)$ contains a subgroup isomorphic to $H$, but $G$ does not, and thus provides a counterexample to the question of Higman. However Hertweck's counterexample is a rather complicated group of or order $2^{21} \cdot 97^{28}$ and derived length 4 and leaves Higman's question totally open for smaller and/or ``easier'' groups. This leads to the following question:\\

{\textit{Subgroup Isomorphism Problem} \ \textbf{(SIP)}}: For which finite groups $U$ does the following statement hold: If $\mathrm{V}(\mathbb{Z}G)$ contains a subgroup isomorphic to $U$ for some finite group $G$, then $G$ contains a subgroup isomorphic to $U$.\\

Not explicitly studied, this was only known to hold for cyclic groups of prime power order, following from a result of Cohn and Livingstone that the exponents of $G$ and $\mathrm{V}(\mathbb{Z}G)$ coincide \cite[Corollary 4.1]{CohnLivingstone}. Z. Marciniak raised the question for the Klein four group at an ICM Satellite Conference in 2006 and W. Kimmerle immediately gave a positive answer \cite{KimmiC2C2}, followed by a positive answer for elementary abelian $p$-groups of rank 2 for odd primes $p$ by M. Hertweck \cite{HertweckCpCp}. Since then no further progress has been made on the question. The name "Subgroup Isomorphism Problem" was coined by W. Kimmerle in several talks and is also recorded in \cite{KimmerleSylow}. The problem appeared for the first time in the literature in \cite[Problem 19]{Ari}. Denote by $C_n$ the cyclic group of order $n.$ The purpose of this article is the proof of the following result:

\begin{theo}\label{SIPC2C4}
Let $G$ be a finite group. If $\mathrm{V}(\mathbb{Z}G)$ contains a subgroup isomorphic to $C_4 \times C_2$, then so does $G$, i.e. the Subgroup Isomorphism Problem holds for $C_4 \times C_2.$
\end{theo} 

Doing just a little more we obtain also the following results, being of independent interest. The investigation of the following Theorem was proposed in \cite[Comment before Example 7]{HertweckCpCp}.

\begin{theo}\label{2Dieder}
Let $G$ be a finite group possessing a dihedral Sylow $2$-subgroup. Then any $2$-subgroup of $\mathrm{V}(\mathbb{Z}G)$ is isomorphic to a subgroup of $G.$ 
\end{theo}

Moreover we obtain the following generalisation of \cite[Proposition 3.4]{AndreasKimmi} and \cite[Proposition 4.7]{KimmerleSylow}.

\begin{prop}\label{Ord8}
Let $G$ be a finite group whose Sylow $2$-subgroup has at most order $8$ and assume moreover that $G$ is not isomorphic to the alternating group of degree $7$. Then any $2$-subgroup $U$ of $\mathrm{V}(\mathbb{Z}G)$ is rationally conjugate to a subgroup of $G$, i.e. there exists a unit $x$ in the rational group algebra $\mathbb{Q}G$ and a subgroup $P$ of $G$ such that $x^{-1}Ux = P$.
\end{prop}

\section{Preliminaries}

The following Lemma from elementary group theory is the key observation to allow the proof of Theorem \ref{SIPC2C4}.
\begin{lemma}\label{2Gruppen}
Let $P$ be a finite 2-group such that $P$ contains no subgroup isomorphic to $C_4 \times C_2.$ Then $P$ is one of the following: Elementary-abelian, cyclic, a generalized quaternion group, a dihedral group or a semidihedral group.
\end{lemma}

\begin{rem}
Note that semidihedral groups are called "quasidihedral" groups by many authors. A presentation of a semidihedral group of order $2^{n+1}$ is given by
\[ \langle h, a \ | \ h^{2^n} = a^2 = 1, h^a = h^{-1 + 2^{n-1}} \rangle. \] 
For definitions of all groups mentioned in the Lemma see \cite[5.3]{Kurzweil}.
\end{rem}

\begin{proof}  If $P$ is abelian, the statement is clear and if $P$ is not abelian and contains only one involution, $P$ is a quaternion group \cite[5.3.7]{Kurzweil}. So assume $P$ is not abelian and contains at least two involutions. Since $P$ then contains an element of order $4$, say $g$, we have $C_2 \cong Z(P) = \langle g^2 \rangle.$ Assume $P$ possesses a non-cyclic abelian normal subgroup $M$, which must be elementary-abelian, and let $a$ be an element of $M$ which is not central in $P$. Then $C_P(M) = N$ is also normal, elementary-abelian and contains $g^2$. If the order of $N$ is $4$, we have $M = N = C_P(a) \cong C_2 \times C_2$ and $P$ is a dihedral or semidihedral group by \cite[5.3.10]{Kurzweil}. So assume $|N| \geq 8$ and let $b$ be an element of $N$ which is not central in $P$ such that $ab$ is also not central in $P.$ Since $a^ga$ commutes with $g$ we have $a^g = ag^2$ and in the same way $b^g = bg^2.$ But then $(ab)^g = ab$, contradicting the existence of $N$ and $M.$ Hence there is no non-cyclic abelian normal subgroup in $P$ and $P$ is dihedral or semidihedral by \cite[5.3.9]{Kurzweil}. 
\end{proof}

Today a main open question concerning torsion units in $\mathbb{Z}G$ is the Zassenhaus Conjecture. It states that for every finite group $G$ and every torsion unit $u \in \mathrm{V}(\mathbb{Z}G)$ there exist a unit $x$ in the rational group algebra $\mathbb{Q}G$ and an element $g \in G$ such that $x^{-1}ux = g.$ In this case one says that $u$ and $g$ are rationally conjugate.
 The main notion to study torsion subgroups of $\mathbb{Z}G$ are the so called partial augmentations. For a conjugacy class $x^G$ in $G$ and an element $u = \sum\limits_{g \in G} z_gg$ in $\mathbb{Z}G$ the sum
\[\varepsilon_x(u) = \sum\limits_{g \in x^G} z_g\]
is called the partial augmentation of $u$ with respect to $x.$ The connection between partial augmentations and the Zassenhaus Conjecture is established by \cite[Theorem 2.5]{MarciniakRitterSehgalWeiss}: A torsion unit $u \in \mathrm{V}(\mathbb{Z}G)$ of order $n$ is rationally conjugate to an element of $G$ if and only if $\varepsilon_x(u^d) \geq 0$ for all $x \in G$ and divisors $d$ of $n.$ If $u$ is a torsion unit and $x$ is an element of $G$ of order not dividing the order of $u$ then $\varepsilon_x(u) = 0$ \cite[Theorem 2.3]{HertweckBrauer}. Moreover $\varepsilon_z(u) = 0$ for every central element $z$ of $G$, unless $u = z$ by the Berman-Higman Theorem \cite[Proposition 1.5.1]{GRG1}. These facts will be used in the following without further mention.

A nowadays well known method to study the torsion units in integral group rings was introduced in \cite{LutharPassiA5} for ordinary characters and later generalized in \cite{HertweckBrauer} for Brauer characters and became known in the last years as the HeLP-method. Though developed to study finite cyclic subgroups of $\mathrm{V}(\mathbb{Z}G)$ its idea does also work for non-cyclic groups and was applied for ordinary characters a few times (namely in \cite{Hoefert}, \cite{HertweckHoefertKimmerle}, \cite{HertweckCpCp} and \cite{PSL2p3}). Since for non-cyclic groups Brauer characters have not been used so far, the following Lemma seems necessary. Note that for a prime $p$ and a $p$-Brauer character $\varphi$ of a finite group $G$ with corresponding representation $D$ we can extended the domain of $\varphi$ to the torsion units of $\mathrm{V}(\mathbb{Z}G)$ of order prime to $p$ by linearly extending $D$ to $\mathbb{Z}G$, see \cite[\textsection 3]{HertweckBrauer}.

\begin{lemma}\label{ModHeLP}
Let $G$ be a finite group, $U$ a finite subgroup of $\mathrm{V}(\mathbb{Z}G)$ of order $n$ and $p$ a prime not dividing $n.$ Let $\chi$ be the extension of a  $p$-Brauer character of $G$ to the $p$-regular elements of $\mathrm{V}(\mathbb{Z}G)$ and let $\psi$ be an ordinary character of $U.$ 
\begin{itemize}
\item[a)] \cite[Theorem 3.2]{HertweckBrauer} For every $u \in U$ we have $\chi(u) = \smashoperator[r]{\sum\limits_{\substack{x^G \\ x \ p-{\rm{regular}}}}} \varepsilon_x(u)\chi(x).$
\item[b)] $\frac{1}{n} \sum\limits_{u \in U} \chi(u)\psi(u^{-1})$ is a non-negative integer.
\end{itemize}
\end{lemma}

\begin{proof} Since $p$ is not dividing $n$ there is an ordinary character $\chi'$ of $U$ such that $\chi'(u) = \chi(u)$ for every $u \in U$ by \cite[Corollary 18.11]{CR1}. So 
\[ \frac{1}{n} \sum\limits_{u \in U} \chi(u)\psi(u^{-1}) = \langle \chi', \psi \rangle_U \]
equals the scalar product of two ordinary characters and thus it is a non-negative integer.
\end{proof}

The proof of Theorem \ref{SIPC2C4} will be achieved by checking all possible cases for Sylow $2$-subgroups given by Lemma \ref{2Gruppen}. The knowledge so far is recorded in \cite{KimmerleSylow}.

\begin{prop}\label{Kimmi} \cite[Theorem 4.1, Proposition 4.8]{KimmerleSylow}
Let $G$ be a finite group whose Sylow $2$-subgroup is either abelian or a (generalized) quaternion group. Then any $2$-subgroup of $\mathrm{V}(\mathbb{Z}G)$ is isomorphic to a subgroup of $G.$
\end{prop}

Groups possessing dihedral or semidihedral Sylow $2$-subgroups were classified in a series of articles \cite{GorensteinWalter} (see also \cite{Bender} for an alternative proof) and \cite{AlperinBrauerGorenstein}. These will be the basis for the rest of the proofs. We will denote by $O_{2'}(G)$ the biggest normal subgroup of odd order of a group $G.$ The following Lemma follows from \cite[Lemma 2.1, Theorem 2.2]{DokuchaevJuriaans}. 
 
\begin{lemma}\label{Normalteiler}
Let $G$ be a finite group, $N$ a normal subgroup of $G$ and $U$ a finite subgroup of $\mathrm{V}(\mathbb{Z}G)$ such that the orders of $N$ and $U$ are coprime. Then $U$ is rationally conjugate to a subgroup of $G$ if and only if the image of $U$ in $\mathrm{V}(\mathbb{Z}(G/N))$ is rationally conjugate a subgroup of $G/N$. Moreover $\mathrm{V}(\mathbb{Z}(G/N))$ possesses a subgroup isomorphic to $U.$
\end{lemma} 
 
In the proofs of the main results we will need a lot of information about finite classical groups. We are not going to cite the original literature here of Dickson, Schur, Dieudonn\'e, Brauer, just to mention a few, but literature which is easier accessible and better to understand nowadays. We will also almost not use any character tables accessible in the literature, but natural representations of the groups involved.
 
\section{Proof of main results} 
 
The first Lemma concerns a big class of the groups involved in the results. 
 
\begin{lemma}\label{PSL}
Let $G$ be a subgroup of ${\operatorname{P\Gamma L}}(2,q) = {\rm{Aut}}(\operatorname{PSL}(2,q))$ containing $\operatorname{PSL}(2,q)$ or $\operatorname{PGL}(2,q)$ as a normal subgroup of odd index, where $q$ is an odd prime power. Then units of order $2$ and $4$ in $\mathrm{V}(\mathbb{Z}G)$ are rationally conjugate to elements of $G$. In particular squares of units of order $4$ from $\mathrm{V}(\mathbb{Z}G)$ are rationally conjugate to involutions in $\operatorname{PSL}(2,q)$.
\end{lemma} 

\begin{proof}
The following facts are given in \cite[Lemma 3.1]{GorensteinWalter}. The Sylow $2$-subgroup of $G$ is a dihedral group or an elementary abelian-group of rank $2$. The group $\operatorname{PSL}(2,q)$ possesses exactly one conjugacy class of involutions, say $2a$, while the group $\operatorname{PGL}(2,q)$ possesses exactly two conjugacy classes of involutions, the one not lying in $\operatorname{PSL}(2,q)$ will be called $2b.$ If $\operatorname{PGL}(2,q)$ is not a normal subgroup of $G$, the conjugacy class $2b$ can be ignored in the following computations. If $G$ possesses no elements of order $4$, then $\operatorname{PGL}(2,q)$ is not a subgroup of $G$ and there exists exactly one conjugacy class of involutions in $G$. Since for $u \in \mathrm{V}(\mathbb{Z}G)$ of order $2$ and $x \in G$ of order different from $2$ we have $\varepsilon_x(u) = 0$, this immediately yields the desired result. So assume that $G$ contains elements of order $4$ and denote by $4a$ the conjugacy class of these elements. There is exactly one such class, since this is the case in a dihedral group. The class $4a$ may lie in $\operatorname{PSL}(2,q)$ or outside of it and these two cases will be partly separated in the following computations. 

If $\operatorname{PGL}(2,q)$ is a subgroup of $G$, then $G/\operatorname{PSL}(2,q)$ has a Sylow $2$-subgroup of order $2$. Hence it possesses a normal $2$-complement by \cite[7.2.2]{Kurzweil}. Thus the group $G$ maps onto a cyclic group of order $2$ such that $\operatorname{PSL}(2,q)$ is in the kernel of this map and elements of $\operatorname{PGL}(2,q)$ outside of $\operatorname{PSL}(2,q)$ are not. Hence the group $G$ possesses a 1-dimensional representation mapping $\operatorname{PSL}(2,q)$ to 1 and elements of $\operatorname{PGL}(2,q)$ outside of $\operatorname{PSL}(2,q)$ to $-1$. The character corresponding to this representation will be called $\chi.$ If $\operatorname{PGL}(2,q)$ is not a subgroup of $G$, the character $\chi$ can be ignored.

Moreover $\operatorname{GL}(2,q)$ is acting via conjugation on the $2\times 2$-matrices over $\mathbb{F}_q$ with trace $0$. The kernel of this action is exactly the centre of $\operatorname{GL}(2,q)$, thus inducing a 3-dimensional representation of $\operatorname{PGL}(2,q).$ The character corresponding to the induced representation on $G$ and its twist with $\chi$ will be called $\psi_+$ and $\psi_-.$ Note that these characters take only integral values on the classes we are interested in. Let $N$ be the normal subgroup of $G$ which is $\operatorname{PSL}(2,q)$ or $\operatorname{PGL}(2,q)$ such that $[G : N] = m$ is odd. The characters described above take the values given in Table \ref{4aInPSL} and Table \ref{4aNotInPSL} on the classes of interest.

\begin{table}[h]
\centering
\begin{tabular}{r|cccc} 
 & $1a$ & $2a$ & $4a$ & $2b$ \\ \hline  \\[-2ex]
$\chi$ & $ 1 $ & $ 1 $ & $1$ & $-1$ \\ 
 $\psi_+$ & $3m$ & $-m$ & $m$ & $-m$ \\ 
 $\psi_-$ & $3m$ & $-m$ & $m$ & $m$ 
\end{tabular}
 \smallskip
 \caption{Some characters of $G$, if elements of order $4$ are in $\operatorname{PSL}(2,q)$.}\label{4aInPSL}
\end{table}

\begin{table}[h]
\centering
\begin{tabular}{r|cccc} 
 & $1a$ & $2a$ & $4a$ & $2b$ \\ \hline  \\[-2ex]
$\chi$ & $ 1 $ & $ 1 $ & $-1$ & $-1$ \\ 
 $\psi_+$ & $3m$ & $-m$ & $m$ & $-m$ \\ 
 $\psi_-$ & $3m$ & $-m$ & $-m$ & $m$ 
\end{tabular}
 \smallskip
 \caption{Some characters of $G$, if elements of order $4$ are not in $\operatorname{PSL}(2,q)$.}\label{4aNotInPSL}
\end{table}

Let $u \in \mathrm{V}(\mathbb{Z}G)$ be of order $2$. We have $\varepsilon_{2a}(u) + \varepsilon_{2b}(u) = 1.$ Moreover from $\chi$ we obtain 
\[\chi(u) = \varepsilon_{2a}(u) - \varepsilon_{2b}(u) \in \{\pm 1\}.\]

Thus $(\varepsilon_{2a}(u), \varepsilon_{2b}(u)) \in \{(1,0), (0,1)\}$, proving the claim for units of order $2$.

Now let $u \in \mathrm{V}(\mathbb{Z}G)$ be of order $4$. Note that since $\chi(u)$ is integral, $\chi(u) \in \{\pm 1\}$ and thus $\chi(u^2) = 1$ meaning that $u^2$ is rationally conjugate to elements in $2a.$ We have 
\begin{align}\label{A}
\varepsilon_{2a}(u) + \varepsilon_{2b}(u) + \varepsilon_{4a}(u) = 1.
\end{align}
Let $D$ be a representation corresponding to $\psi_+$ or $\psi_-.$ Since $\psi_-(u^2) = \psi_+(u^2) = \psi_+(2a) = -m$ the eigenvalues of $D(u^2)$, with multiplicities, are $2m$ times $-1$ and $m$ times $1$. Thus $2m$ of the eigenvalues of $D(u)$ are $i$ or $-i$, where $i$ denotes a primitive $4$th root of unity in $\mathbb{F}_{q^2}$. But since $\psi_+$ and $\psi_-$ do only take real values, $m$ eigenvalues of $D(u)$ must be $i$ and $m$ must be $-i.$ In particular $-m \leq \psi_{+}(u) \leq m$ and $-m \leq \psi_{-}(u) \leq m.$

Assume firstly that $4a$ is lies in $\operatorname{PSL}(2,q).$ Then using $\chi$ we obtain
\[\chi(u) = \varepsilon_{2a}(u) - \varepsilon_{2b}(u) + \varepsilon_{4a}(u) \in \{ \pm 1\}.\]
Subtracting \eqref{A} this gives $\varepsilon_{2b}(u) \in \{0, 1 \}.$  Using $\psi_+$ and $\psi_-$ we obtain
\[-m \leq -m\varepsilon_{2a}(u) + m\varepsilon_{4a}(u) - m\varepsilon_{2b}(u) \leq m\]
and
\[-m \leq -m\varepsilon_{2a}(u) + m\varepsilon_{4a}(u) + m\varepsilon_{2b}(u) \leq m\]
respectively. Adding or subtracting \eqref{A} gives $\varepsilon_{4a}(u) \in \{0,1\}$ and $\varepsilon_{2a}(u) \in \{0, 1\}$ respectively. Since $\varepsilon_{4a}(u) \not \equiv 0 \mod 2$ and $\varepsilon_{2a}(u) \equiv \varepsilon_{2b}(u) \equiv 0 \mod 2$ by \cite[Theorem 4.1]{CohnLivingstone} this implies the equality $(\varepsilon_{2a}(u), \varepsilon_{2b}(u), \varepsilon_{4a}(u)) = (0,0,1)$.

Now assume $4a$ is not in $\operatorname{PSL}(2,q).$ Doing analogues computations to the above $\chi$ gives $\varepsilon_{2a}(u) \in \{0, 1 \}$, while $\psi_+$ and $\psi_-$ give $\varepsilon_{2b}(u) \in \{0, 1 \}$ and $\varepsilon_{4a}(u) \in \{0, 1\}$ respectively. By \cite[Theorem 4.1]{CohnLivingstone} however $\varepsilon_{4a}(u) \not \equiv 0 \mod 2$ and thus we have $(\varepsilon_{2a}(u), \varepsilon_{2b}(u), \varepsilon_{4a}(u)) = (0,0,1),$ proving that units of order $4$ are rationally conjugate to elements of $G$.
\end{proof} 
 
We are now ready to prove the main results. \\
 
\textit{Proof of Theorem \ref{2Dieder}:} By Lemma \ref{Normalteiler} we may assume that $O_{2'}(G) = 1.$  So by \cite[Theorem of Gorenstein and Walter]{Bender} one of the following three cases may be assumed:
\begin{itemize}
\item[(i)] $G$ is a subgroup of ${\operatorname{P\Gamma L}}(2,q)$ containing $\operatorname{PSL}(2,q)$ or $\operatorname{PGL}(2,q)$ as a normal subgroup of odd index $m$. Here $q$ denotes an odd prime power.
\item[(ii)] $G$ is the alternating group of degree 7.
\item[(iii)] $G$ is a 2-group.
\end{itemize} 
In case (iii) the result follows from \cite[Theorem 1]{Weiss88} and case (ii) has been handled in \cite[Example 7]{HertweckCpCp}. So (i) is the only remaining case and we are exactly in the situation of Lemma \ref{PSL}. We may moreover assume that $G$ contains elements of order $4$, since otherwise the result follows from Proposition \ref{Kimmi}. We will frequently use Lemma \ref{PSL} and the characters from the Tables \ref{4aInPSL} and \ref{4aNotInPSL}.

Assume $\mathrm{V}(\mathbb{Z}G)$ contains a subgroup $U$ isomorphic to $C_4 \times C_2.$ Note that $U$ contains three involutions and four elements of order $4$ and $\psi_+$ takes the value $-m$ on involutions of $U$ while taking the value $m$ on all elements of order $4$ in $U$. Let $\mathbf{1}$ be the trivial character of $U.$ Then by Lemma \ref{ModHeLP} 
\[\frac{1}{8}\sum_{u \in U}\psi_+(u)\mathbf{1}(u^{-1}) = \frac{1}{8}(3m - 3m + 4m) = \frac{m}{2}\]
is a non-negative integer, contradicting the existence of $U$ since $m$ is odd.

Now in view of Lemma \ref{2Gruppen} and the fact that the exponents of $G$ and $\mathrm{V}(\mathbb{Z}G)$ coincide, it only remains to show that there are no elementary-abelian groups of rank bigger then $2$, quaternion or semidihedral groups in $\mathrm{V}(\mathbb{Z}G).$ Since every quaternion and semidihedral group contains a quaternion group of order $8$, it will be enough to prove that $\mathrm{V}(\mathbb{Z}G)$ does not contain neither a subgroup isomorphic to $C_2 \times C_2 \times C_2$ nor $Q_8.$ Let $U$ be an elementary-abelian group of rank $3$. Then $\psi_+(u) = -m$ for every $u \neq 1$ in $U$ and by Lemma \ref{ModHeLP}
\[\frac{1}{8}\sum_{u \in U}\psi_+(u)\mathbf{1}(u^{-1}) = \frac{1}{8}(3m - 7m) = \frac{-m}{2}\]
is a non-negative integer, contradicting the existence of $U.$

Finally let $U \leq \mathrm{V}(\mathbb{Z}G)$ be a quaternion group of order $8$. Note that $U$ contains exactly one involution and six elements of order $4$. Since the involution in $U$ is a square of an element of order $4$ it is rationally conjugate to elements in $2a$. We claim that $4a$ is in $\operatorname{PSL}(2,q)$. Indeed, otherwise 
\[\frac{1}{8}\sum_{u \in U}\chi(u)\mathbf{1}(u^{-1}) = \frac{1}{8}(2 - 6) = \frac{-1}{2},\]
contradicting Lemma \ref{ModHeLP}. So assume that $4a$ lies in $\operatorname{PSL}(2,q).$ Note that since the $2$-Sylow subgroup of $\operatorname{PSL}(2,q)$ is dihedral and $|\operatorname{PSL}(2,q)| = \frac{(q-1)q(q+1)}{2}$ this is equivalent to $q \equiv \pm 1 \mod 8$. Any irreducible ordinary character of degree $q \pm 1$ of $\operatorname{PSL}(2,q)$ may be extended to a character of $\operatorname{PGL}(2,q)$ by \cite[Lemma 4.5]{White}, meaning it has the same values on the elements lying in $\operatorname{PSL}(2,q).$ If $\operatorname{PGL}(2,q)$ is not a normal subgroup of $G$, this fact may just be ignored. In any case every ordinary irreducible character $\eta'$ of degree $q \pm 1$ of $\operatorname{PSL}(2,q)$ can be induced to a character $\eta$ of $G$ such that $\eta(1a) = m\cdot \eta'(1a), \eta(2a) = m\cdot \eta'(2a)$ and $\eta(4a) = m\cdot \eta'(4a).$ Since in $\operatorname{PGL}(2,q)$ every element is conjugate to its inverse the extension of $\eta'$ to $\operatorname{PGL}(2,q)$ is real-valued. Moreover the Schur index of this extension is $1$ by \cite[Theorem 2(a)]{Gow} and hence it is affordable by a real representation and then so is $\eta$.

We will apply the same arguments as in the last paragraph of the proof of \cite[Theorem 2.1]{HertweckHoefertKimmerle}. Let $\varepsilon \in \{ \pm 1 \}$ such that $q \equiv \varepsilon \mod 4$. The ordinary character table of $\operatorname{PSL}(2,q)$ was first computed independently by Schur and Jordan and is also given in \cite[Table 1]{HertweckHoefertKimmerle}. Let $\eta'$ be an irreducible ordinary character of $\operatorname{PSL}(2,q)$ of degree $q + \varepsilon$ such that $\eta'(2a) = -2\varepsilon$ and $\eta'(4a) = 0$. Let $\eta$ be the character of $G$ obtained from $\eta'$ the way described in the preceding paragraph. Then $\eta(1a) = m(q + \varepsilon), \ \eta(2a) = -2m\varepsilon$ and $\eta(4a) = 0$. Let $\lambda$ be the $2$-dimensional irreducible character of $U$, so $\lambda$ takes the value $-2$ on the involution in $U$ and the value $0$ on elements of order $4$. Taking the scalar product of $\eta$ and $\lambda$ with respect to $U$ we obtain
\[\langle \lambda, \eta \rangle_U = \frac{1}{8}(2m(q+\varepsilon) + 4m\varepsilon) = \frac{m(q-\varepsilon)}{4} + m\varepsilon. \] 
Since $q \equiv \pm 1 \mod 8$ the number $\frac{q-\varepsilon}{4}$ is even, thus $\langle \lambda, \eta \rangle_U$ is odd. By the preceding paragraph $\eta$ is affordable by a real representation. Linearly extending this representation we obtain a real representation of $U$ with character $\eta|_U$. Since the Frobenius-Schur indicator of $\lambda$ is $-1$ it is not affordable by a real representation \cite[XI, Theorem 8.3]{HuppertIII}. So by \cite[XI, Theorem 8.9]{HuppertIII} the integer $ \langle \lambda, \eta \rangle_U$ must be even, a contradiction. Thus $\mathrm{V}(\mathbb{Z}G)$ does not contain a quaternion group of order $8$ and this finishes the proof of the theorem.  \hfill $\qed$ \\

\textit{Proof of Proposition \ref{Ord8}:} By \cite[Proposition 4.7]{KimmerleSylow} we may assume that the Sylow $2$-subgroup of $G$ is not abelian, thus it is a dihedral or quaternion group of order $8$. First assume that the Sylow $2$-subgroup is dihedral. By Lemma \ref{Normalteiler} we may assume that $O_{2'}(G) = 1$ and so again by \cite[Theorem of Gorenstein and Walter]{Bender} one of the three cases given in the beginning of the proof of Theorem \ref{2Dieder} may be assumed. If $G$ is a $2$-group the result is well known and follows also from \cite[Theorem 1]{Weiss88}. Since by assumption $G$ is not isomorphic to an alternating group of degree $7$, it is a subgroup of ${\operatorname{P\Gamma L}}(2,q)$ containing $\operatorname{PSL}(2,q)$ or $\operatorname{PGL}(2,q)$ as a normal subgroup of odd index with some odd prime power $q.$ Let $U$ be a $2$-subgroup of $\mathrm{V}(\mathbb{Z}G)$. Then $U$ is isomorphic to a subgroup of $G$ by Theorem \ref{2Dieder}. Let the conjugacy classes of elements of order $2$ and $4$ be named as in the proof of Lemma \ref{PSL}. If $\operatorname{PGL}(2,q)$ is not a subgroup of $G$, Lemma \ref{PSL} implies that any isomorphism between $U$ and a subgroup $S$ of $G$ preserves values of all ordinary irreducible characters of $G$ and then $U$ and $S$ are rationally conjugate by \cite[Lemma 4]{Valenti}.

So assume that $\operatorname{PGL}(2,q)$ is a subgroup of $G$. Note that in this case elements of order $4$ in $G$ do not lie in $\operatorname{PSL}(2,q)$ since otherwise the Sylow $2$-subgroup of $G$ would have more then $8$ elements. We thus may assume that $G$ has characters as given in Table \ref{4aNotInPSL}. If $U$ is cyclic, then it is rationally conjugate to a subgroup of $G$ by Lemma \ref{PSL}. If $U$ is elementary-abelian of rank $2$, either all elements of $U$ are rationally conjugate to elements in $2a$ or exactly two elements are rationally rationally conjugate to elements of $2b$, since otherwise $\frac{1}{4}\sum\limits_{u \in U}\chi(u)$ is not a non-negative integer, contradicting Lemma \ref{ModHeLP}. Thus $U$ is rationally conjugate to a subgroup of $G$ by \cite[Lemma 4]{Valenti}, since in both cases we can find a subgroup $S$ of $G$ and an isomorphism between $U$ and $S$ preserving values on all ordinary characters of $G$.

So finally assume that $U$ is isomorphic to a dihedral group of order $8$. The central involution of $U$ is rationally conjugate to elements in $2a$ by Lemma \ref{PSL} since it is a square of an unit of order $4$. Since $\frac{1}{8}\sum\limits_{u \in U}\chi(u)$ is an integer, one of the non-central involutions in $U$ is rationally conjugate to elements of $2b$ and the character values of $\chi$ then determine for every involution in $U$, whether it is rationally conjugate to an element of $2a$ or $2b$. This allows again to construct an isomorphism between $U$ and a Sylow $2$-subgroup of $G$ preserving character values and thus $U$ is rationally conjugate to Sylow $2$-subgroups of $G$ by \cite[Lemma 4]{Valenti}.

It remains to consider the case that the Sylow $2$-subgroup of $G$ is an quaternion group of order $8$. By \cite[Proposition 4.5]{KimmerleSylow} it is enough to show that units of order $2$ and $4$ are rationally conjugate to elements of $G$. By the famous Theorem of Brauer and Suzuki \cite{BrauerSuzuki}, $G$ contains a central involution, say $g$, and by the Berman-Higman Theorem $g$ is the only involution in $\mathrm{V}(\mathbb{Z}G)$. So assume $u \in \mathrm{V}(\mathbb{Z}G)$ has order $4$. Then $\varepsilon_g(u) = 0$ by the Berman-Higman Theorem. If $G$ contains only one conjugacy class of elements of order $4$, this implies that $u$ is rationally conjugate to an element of $G$. So assume that $G$ has more then one conjugacy class of elements of order $4$ and that $u$ is not rationally conjugate to an element of $G$. Different conjugacy classes of elements of order $4$ in $G$ map to different conjugacy classes of involutions in $G/\langle g \rangle$. Thus, if $u$ has non-trivial partial augmentations, the image of $u$ in $\mathrm{V}(\mathbb{Z}(G/\langle g \rangle))$ is an involution with non-trivial partial augmentations. Such an involution does not exist by \cite[Proposition 3.4]{AndreasKimmi} and hence $u$ is rationally conjugate to an element of $G$ and the Proposition is proved. \hfill $\qed$

\begin{rem}
As already remarked in \cite[Example 7]{HertweckCpCp} it is not known, whether units of order $4$ in $\mathrm{V}(\mathbb{Z}A_7)$ are rationally conjugate to elements of the group base. If one could prove that this is the case, then the exception of $A_7$ from Proposition \ref{Ord8} would not be necessary.
\end{rem}

\textit{Proof of Theorem \ref{SIPC2C4}:} Let $G$ be a group not containing a subgroup isomorphic to $C_4 \times C_2$ such that $\mathrm{V}(\mathbb{Z}G)$ contains a subgroup $U = \langle t \rangle \times \langle s \rangle \cong C_4 \times C_2.$ By Lemma \ref{2Gruppen} we may assume that the Sylow $2$-subgroup $P$ of $G$ is abelian, (generalized) quaternion, dihedral or semidihedral. By Proposition \ref{Kimmi} and Theorem \ref{2Dieder} only the case where $P$ is semidihedral remains and by Lemma \ref{Normalteiler} we may assume $O_{2'}(G) = 1.$ The groups of interest are classified in \cite{AlperinBrauerGorenstein}, but are not given in a single theorem there. By \cite[II, 1., Proposition 1]{AlperinBrauerGorenstein} four different cases may appear. These cases get their names in \cite[II, 2., Definition 1]{AlperinBrauerGorenstein}. So $G$ is one of the following:
\begin{itemize}
\item[(i)] A $QD$-group. $G$ possesses exactly one conjugacy class of involutions and one conjugacy class of elements of order $4$.
\item[(ii)] A $Q$-group. $G$ possesses exactly two conjugacy classes of involutions and one conjugacy class of elements of order $4$.
\item[(iii)] A $D$-group. $G$ possesses exactly one conjugacy class of involutions and two conjugacy classes of elements of order $4$.
\item[(iv)] A 2-group.
\end{itemize}
In case (iv) \cite[Theorem 1]{Weiss88} gives the result, so we have to consider the first three cases.\\

\textit{Case 1:} $G$ is a $QD$-group. Denote by $2a$ the conjugacy class of involutions in $G$ and by $4a$ the class of elements of order $4$.

By \cite[II, 2., Proposition 2]{AlperinBrauerGorenstein}, $G$ possesses a simple normal subgroup $N$ of odd index $m$ and by \cite[Third Main Theorem]{AlperinBrauerGorenstein} the normal subgroup $N$ is either $\operatorname{PSL}(3,q)$ with $q \equiv -1 \mod 4$ or $\operatorname{PSU}(3,q)$ with $q \equiv 1 \mod 4$, for some prime power $q$, or the Mathieu group $M_{11}$. Note that $\operatorname{PSU}(3,q)$ is actually defined over $\mathbb{F}_{q^2}$, the field with $q^2$ elements. The groups $\operatorname{SL}(3,q)$ and $\operatorname{SU}(3,q)$ act on the homogeneous polynomials of degree 3 in three commuting variables $x, y, z$ over the field $\mathbb{F}_q$ and $\mathbb{F}_{q^2}$ respectively. See e.g. \cite[pp. 14 - 16]{AlperinLocal} for a discussion of the analogues action of $\operatorname{SL}(2,q).$ The centres of the groups are in the kernel of the action and thus this action supplies a $10$-dimensional representation of $\operatorname{PSL}(3,q)$ and $\operatorname{PSU}(3,q)$ respectively. Let $\chi'$ be the character of this representation. Elements of $\operatorname{SL}(3,q)$ and $\operatorname{SU}(3,q)$ projecting to elements in the classes $2a$ and $4a$ in $G$ are given by $A = \left(\begin{smallmatrix} 1 & 0 & 0 \\ 0 & -1 & 0 \\ 0 & 0 & -1 \end{smallmatrix}\right)$ and $B = \left(\begin{smallmatrix} 1 & 0 & 0 \\ 0 & 0 & 1 \\ 0 & -1 & 0 \end{smallmatrix}\right)$ where the Gram matrix of the underlying unitary form is taken to be the identity matrix. I.e. we understand $\operatorname{SU}(3,q)$ to be those matrices of $\operatorname{GL}(3,q^2)$ having determinant $1$ and leaving the unitary form
\[\alpha: \mathbb{F}_{q^2}^3 \times \mathbb{F}_{q^2}^3 \rightarrow \mathbb{F}_{q^2}, \ \ (x,y) \mapsto \sigma(x)^t \left(\begin{smallmatrix}1 & 0 & 0 \\ 0 & 1 & 0 \\ 0 & 0 & 1 \end{smallmatrix} \right)y \]
invariant. Here $\sigma$ denotes the automorphism of $\mathbb{F}_{q^2}$ of order $2$. 

The eigenspace of $A$ for the eigenvalue $1$ is spanned by $x^3,xy^2,xz^2,xyz$ while the eigenspace for the eigenvalue $-1$ is spanned by $y^3,z^3,x^2y,x^2z,y^2z,yz^2.$ Denote by $i$ a primitive 4th root of unity in $\mathbb{F}_{q^2}.$ The eigenspaces of $B$ are spanned by:
\begin{itemize}
\item For the eigenvalue $i$ by $y^3 + iz^3, yz^2 + iy^2z, x^2y + ix^2z$,
\item for the eigenvalue $-i$ by $y^3 - iz^3, yz^2 - iy^2z, x^2y - ix^2z$,
\item for the eigenvalue $-1$ by $xy^2 - xz^2, xyz$,
\item for the eigenvalue $1$ by $x^3, xy^2 + xz^2$.
\end{itemize} 
One thus obtains $\chi'(1) = 10$, $\chi'(2a) = -2$ and $\chi'(4a) = 0.$ The Mathieu group $M_{11}$ does also posses an (ordinary) character with the same values, see e.g. \cite{Atlas}. Inducing this character of $N$ to $G$ we get the character given in Table \ref{CharQD}.

\begin{table}[h]
\centering
\begin{tabular}{r|ccc} 
 & $1a$ & $2a$ & $4a$ \\ \hline  \\[-2ex]
$\chi$ & $ 10m $ & $ -2m $ & $0$
\end{tabular}
 \smallskip
 \caption{A character of a $QD$-group $G$.}\label{CharQD}
\end{table}

Unlike in Lemma \ref{PSL} it is not possible to show using only the HeLP-method that units of order $4$ in $\mathrm{V}(\mathbb{Z}G)$ are rationally conjugate to elements of $G$. This may be checked e.g. using a GAP-package implementing the method \cite{HeLPPaper}. See also \cite{KonovalovM11} concerning $M_{11}$. Let $u \in \mathrm{V}(\mathbb{Z}G)$ be of order $4$. 
If $D$ is a representation affording $\chi$, then $\chi(u)$ is integral and the eigenvalues of $D(u)$ are $4$th roots of unity and thus $\chi(u) = \chi(u^3).$ Recall that $C_4 \times C_2 \cong U = \langle t \rangle \times \langle s \rangle$ is the group we want to study.  There are three involutions in $U$ taking the value $-2m$ for $\chi.$ Moreover the elements of order $4$ in $U$ are $t, t^3, st, st^3$ and for them we have $\chi(t) = \chi(t^3)$ and $\chi(st) = \chi(st^3)$. So by Lemma \ref{ModHeLP}
\[\frac{1}{8} \sum_{u \in U} \chi(u) = \frac{1}{8} \left(10m -6m + 2\chi(t) + 2\chi(st)\right)\]
 is a non-negative integer. Note that since $\chi(4a) = 0$ we have $\chi(t) = -2m\varepsilon_{2a}(t)$ and $\chi(st) = -2m\varepsilon_{2a}(st)$. Since $\varepsilon_{2a}(t) \equiv \varepsilon_{2a}(st) \equiv 0 \mod 2$ by \cite[Theorem 4.1]{CohnLivingstone}, this implies $\chi(t) \equiv \chi(st) \equiv 0 \mod 4$. Thus $2\chi(t) \equiv 2\chi(st) \equiv 0 \mod 8$, implying $8 \mid 4m$, a contradiction.\\

\textit{Case 2:} $G$ is a $Q$-group. Denote by $2a$ and $2b$ the conjugacy classes of involutions in $G$ and by $4a$ the class of elements of order $4$. 

By \cite[II, 3., Proposition 2]{AlperinBrauerGorenstein}, $G$ possesses a normal subgroup isomorphic to $\operatorname{SL}(2,q)$ for some odd prime power $q$ and hence by \cite[II, 3., Proposition 3]{AlperinBrauerGorenstein} $G$ possesses a normal subgroup $N$ isomorphic to a group denoted by Alperin, Brauer and Gorenstein as $\operatorname{SL}_k(2,q)$ or $\operatorname{SU}_k(2,q)$ and $[G : N] =m$ is odd. The groups appearing for $N$ are defined in \cite[p. 17 above Lemma 1]{AlperinBrauerGorenstein}. The only relevant fact for us will be that these are subgroups of $\operatorname{GL}(2,q)$ and $\operatorname{GU}(2,q)$ respectively and that they contain all matrices of these groups having determinant $1$ or $-1.$ Representatives of the conjugacy classes of interest in $N$ are given by $\left(\begin{smallmatrix} -1 & 0 \\  0 & -1 \end{smallmatrix}\right)$, $\left(\begin{smallmatrix} 1 & 0 \\  0 & -1 \end{smallmatrix}\right)$ and $\left(\begin{smallmatrix} 0 & -1 \\ 1 & 0 \end{smallmatrix}\right)$ where again the Gram matrix of the underlying unitary form is taken to be the identity matrix. Inducing the character of $N$ given by the determinant map to a character of $G$ turns out to be already enough. The values of this character $\chi$ are given in Table \ref{CharQ}.   

\begin{table}[h]
\centering
\begin{tabular}{r|cccc} 
 & $1a$ & $2a$ & $2b$ & $4a$ \\ \hline  \\[-2ex]
$\chi$ & $ m $ & $ m $ & $-m$ & $m$ 
\end{tabular}
 \smallskip
 \caption{A character of a $Q$-group $G.$\label{CharQ}}
\end{table}

\textit{Claim:} Elements of order $2$ and $4$ in $\mathrm{V}(\mathbb{Z}G)$ are rationally conjugate to elements of $G.$ Hence if $u \in \mathrm{V}(\mathbb{Z}G)$ is of order $4$, then $u^2$ is rationally conjugate to elements in $2a.$

Let first $u \in \mathrm{V}(\mathbb{Z}G)$ be of order $2$. If $u$ is not in $2a$, then $\varepsilon_{2a}(u) = 0$ by the Berman-Higman Theorem, since $2a$ consists of a single element which is central in $G$. Thus $u$ is rationally conjugate to an element of $2b$ in this case. If $u$ is of order $4$, again $\varepsilon_{2a}(u) = 0.$ So from $\varepsilon_{2b}(u) + \varepsilon_{4a}(u) = 1$ and the inequalities
\[-m \leq -m\varepsilon_{2b}(u) + m\varepsilon_{4a}(u) \leq m\]
coming from $\chi$ together with the fact that $\varepsilon_{4a} \not \equiv 0 \mod 2$ \cite[Theorem 4.1]{CohnLivingstone} we conclude $(\varepsilon_{2a}(u), \varepsilon_{2b}(u), \varepsilon_{4a}(u)) = (0,0,1).$ 
 
Now let again $C_4 \times C_2 \cong U = \langle t \rangle \times \langle s \rangle$ be a subgroup of $\mathrm{V}(\mathbb{Z}G).$ As $t^2$ is a square of an element of order $4$ it must be the element of $2a$ by the above claim. The other involutions in $U$ must thus be rational conjugates of elements in $2b$ since the element of $2a$, being central, has no conjugates. Then
\[\frac{1}{8}\sum_{u \in U} \chi(u) = \frac{1}{8}(6m-2m) = \frac{m}{2}\]
is a non-negative integer by Lemma \ref{ModHeLP}, a contradiction.\\

\textit{Case 3:} $G$ is a $D$-group. Denote by $2a$ the conjugacy class of involutions in $G$ and by $4a$ and $4b$ the classes of elements of order $4$.

By \cite[II, 3., Proposition 4]{AlperinBrauerGorenstein} we may assume that $G$ is a $\operatorname{PGL}_n(2,3)$, defined in \cite[p. 20, before Lemma 5]{AlperinBrauerGorenstein}, or $A_7$ or it is a subgroup of $\operatorname{P\Gamma L}(2,q) = {\rm{Aut}}(\operatorname{PSL}(2,q))$, for some prime power $q$, and contains one of the groups $\operatorname{PSL}(2,q)$, $\operatorname{PGL}(2,q)$ or $\operatorname{PGL}^*(2,q)$ as a normal subgroup $N$ such that $[G : N] = m$ and $q$ are both odd. In case $G$ is a $\operatorname{PGL}_n(2,3)$ its Sylow $2$-subgroup is not a semidihedral group by \cite[II, 2., Lemma 5]{AlperinBrauerGorenstein}. In case $G$ is $A_7$ or $N$ is $\operatorname{PSL}(2,q)$ or $\operatorname{PGL}(2,q)$, the Sylow $2$-subgroup is also not semidihedral by \cite[II, 2., Lemma 3]{AlperinBrauerGorenstein}. So only the case $N = \operatorname{PGL}^*(2,q)$ remains, this group is defined in \cite[p. 20, above Lemma 4]{AlperinBrauerGorenstein}, and we are going to describe it now.

The group $N = \operatorname{PGL}^*(2,q)$ is a non-split extension of $\operatorname{PSL}(2,q)$ of degree $2$ and does only exist, if $q$ is a square, say $q = r^2.$ The groups $\operatorname{PGL}^*(2,q)$ are Zassenhaus groups and were introduced, to my knowledge, in \cite{ZassenhausKennzeichnung} where they are denoted by $M_q.$ They are also discussed in \cite[XI, Example 1.3c)]{HuppertIII}. Let $\sigma$ be the unique automorphism of $\mathbb{F}_q$ of order $2$. So in particular, the elements of $\mathbb{F}_r$ are fixed by $\sigma.$ By acting entrywise on $\operatorname{GL}(2,q)$ the automorphism $\sigma$ induces an automorphism, also called $\sigma$, of $\operatorname{GL}(2,q)$ and after projection an automorphism $\bar{\sigma}$ of $\operatorname{PSL}(2,q).$ Denote by \ $\bar{}$ \ the natural homomorphism from $\operatorname{GL}(2,q)$ to $\operatorname{PGL}(2,q).$ Then via conjugation every $A \in \operatorname{GL}(2,q)$ induces an automorphism $\bar{A}$ of $\operatorname{PSL}(2,q).$ We understand $\operatorname{PSL}(2,q)$ to be its own group of inner automorphisms. Let $\alpha$ be an element of maximal $2$-power order in $\mathbb{F}_q$ and $i$ a primitive $4$th root of unity in $\mathbb{F}_q$. Set 
\[g = \left(\begin{smallmatrix} 1 & 0 \\ 0 & \alpha \end{smallmatrix}\right), \ \ \operatorname{if} \ \ r \equiv -1 \mod 4\]
and 
\[g = \left(\begin{smallmatrix} 0 & \alpha \\ -1 & 0 \end{smallmatrix}\right), \ \ \operatorname{if} \ \ r \equiv 1 \mod 4. \] 
Note that $\sigma(\alpha) = -\alpha^{-1}$, if $r \equiv -1 \mod 4$, and $\sigma(\alpha) = -\alpha$, if $r \equiv 1 \mod 4.$ Then $\bar{g}\bar{\sigma}$ is an automorphism of order $4$ of $\operatorname{PSL}(2,q)$ such that $(\bar{g}\bar{\sigma})^2 = \overline{\left(\begin{smallmatrix} i & 0 \\ 0 & -i \end{smallmatrix}\right)}$ is an inner automorphism of $\operatorname{PSL}(2,q).$ We set $\operatorname{PGL}^*(2,q) = \langle \operatorname{PSL}(2,q), \bar{g}\bar{\sigma} \rangle$, understanding it as a subgroup of the automorphism group of $\operatorname{PSL}(2,q).$ The concrete choice of $g$ is inspired by the proof of \cite[Lemma 2.3]{GorensteinCentralizers} where the structure of the Sylow $2$-subgroup of $\operatorname{PGL}^*(2,q)$ is analysed. 

Since $G/\operatorname{PSL}(2,q)$ has a Sylow $2$-subgroup of order $2$ it contains a normal $2$-complement by \cite[7.2.2]{Kurzweil}. Thus $G$ maps onto a cyclic group of order $2$ and $\operatorname{PSL}(2,q)$ is in the kernel of this map while elements of $N$ outside of $\operatorname{PSL}(2,q)$ are not. Hence $G$ has a 1-dimensional representation containing $\operatorname{PSL}(2,q)$ in its kernel and mapping elements of $N$ outside of $\operatorname{PSL}(2,q)$ to $-1$. Call the corresponding character $\chi.$ Moreover $\operatorname{SL}(2,q)$ is acting via conjugation on the $2\times 2$-matrices over $\mathbb{F}_q$ having trace $0$ giving a $3$-dimensional representation of $\operatorname{PSL}(2,q).$ This then induces a $6m$-dimensional representation of $G.$ Call the corresponding character $\psi.$ 

Furthermore $\operatorname{GL}(2,q)$ is acting on the 4-dimensional $\mathbb{F}_r$-vector space 
\[\mathcal{H} = \left\{ \left(\begin{smallmatrix} a & c \\ \sigma(c) & b \end{smallmatrix}\right) \ \mid \ a,b \in \mathbb{F}_{r}, c \in \mathbb{F}_q \right\}\]
by $A * H = \sigma(A)^t H A$ for $A \in \operatorname{GL}(2,q)$ and $H \in \mathcal{H}$, where $X^t$ denotes the transpose of a matrix $X.$ The kernel of this operation is $\left\{ \left(\begin{smallmatrix} x & 0 \\ 0 & x \end{smallmatrix}\right) \ | \ x \in \mathbb{F}_q, x\sigma(x) = 1 \right\}.$ Since the centre of $\operatorname{SL}(2,q)$ is contained in this kernel, we obtain a 4-dimensional $\mathbb{F}_r$-representation of $\operatorname{PSL}(2,q)$ which we want to extend to a representation of $N.$ Also $\sigma$ is acting on $\mathcal{H}$ by entry-wise application. Let $H = \left(\begin{smallmatrix} a & c \\ \sigma(c) & b \end{smallmatrix}\right)$ be an arbitrary element of $\mathcal{H}.$ In case $r \equiv -1 \mod 4$ the element $g \sigma$, as an element of the semilinearities $\Gamma\operatorname{L}(2,q)$, has order $4$ and the action of $\langle g\sigma \rangle$ on $\mathcal{H}$ has trivial kernel. Thus setting 
\[\bar{g}\bar{\sigma} * H = \left(\begin{smallmatrix} a & \alpha\sigma(c) \\ \sigma(\alpha) c  & -b \end{smallmatrix}\right)\]
extends the action of $\operatorname{PSL}(2,q)$ to an action of $N$ providing a $4$-dimensional $\mathbb{F}_r$-representation of $N.$ In case $r \equiv 1 \mod 4$ the order of the semilinearity $g\sigma$ is twice the order of $\alpha$ while the order of $\bar{g}\bar{\sigma}$ is $4$. We have 
\[g\sigma * H = \left(\begin{smallmatrix} b & -\alpha\sigma(c) \\ -\sigma(\alpha)c & -\alpha^2 a \end{smallmatrix}\right)\]
so $(g \sigma)^4 * H = \alpha^4 H$. Hence viewing a basis of $\mathcal{H}$ as a basis of the $4$-dimensional $\mathbb{F}_q$-vector space $\mathbb{F}_q \otimes_{\mathbb{F}_r} \mathcal{H}$ and setting 
\[\bar{g}\bar{\sigma} * H = \alpha^{-1} \left(\begin{smallmatrix} b & -\alpha \sigma(c) \\ -\sigma(\alpha) c & -\alpha^2 a \end{smallmatrix}\right) = \left(\begin{smallmatrix} \alpha^{-1}b & -\sigma(c) \\ c & -\alpha a \end{smallmatrix}\right)\]
we obtain a $4$-dimensional $\mathbb{F}_q$-representation of $N.$ Let $\eta$ be the character corresponding to the induced representation of $G$. Computing the eigenvalues of these actions one obtains the character values given in Table \ref{CharD}.

\begin{table}[h]
\centering
\begin{tabular}{r|cccc} 
 & $1a$ & $2a$ & $4a$ & $4b$ \\ \hline  \\[-2ex]
$\chi$ & $ 1 $ & $ 1 $ & $1$ & $-1$ \\ 
$\psi$ & $ 6m $ & $ -2m $ & $2m$ & $0$ \\
$\eta$ & $ 4m $ & $ 0 $ & $-2m$ & $0$ \\
\end{tabular} 
 \smallskip
 \caption{Some characters of a $D$-group $G.$\label{CharD}}
\end{table} 

\textit{Claim:} Elements of order $2$ and $4$ in $\mathrm{V}(\mathbb{Z}G)$ are rationally conjugate to elements of $G.$ Hence especially an element of order $4$ is rationally conjugate to its inverse. 

This is clear for elements of order $2$, so let $u \in \mathrm{V}(\mathbb{Z}G)$ be of order $4$. From $\chi$ we obtain the equation 
\[\varepsilon_{2a}(u) + \varepsilon_{4a}(u) - \varepsilon_{4b}(u) \in \{ \pm 1 \},\]
and adding $\varepsilon_{2a}(u) + \varepsilon_{4a}(u) + \varepsilon_{4b}(u) = 1$ onto this gives $\varepsilon_{4b}(u) \in \{0,1\}.$ Since $u^2$ is rationally conjugate to an element in $2a$ under a representation affording $\eta$ the unit $u^2$ has $2m$ times the eigenvalue $1$ and $2m$ times the eigenvalue $-1.$ This implies 
\[-2m \leq \eta(u) = -2m\varepsilon_{4a}(u) \leq 2m\]
and thus $\varepsilon_{4a}(u) \in \{-1,0,1\}.$ Under a representation affording $\psi$ the involution $u^2$ has $4m$ times the eigenvalue $-1$ and $2m$ times the eigenvalue $1$. Thus 
\[-2m \leq \psi(u) = -2m\varepsilon_{2a}(u) + 2m\varepsilon_{4a}(u) \leq 2m,\]
implying $-\varepsilon_{2a}(u) + \varepsilon_{4a}(u) \in \{-1,0,1\}.$ Together with the fact that $\varepsilon_{2a}(u) \equiv 0 \mod 2$ and $\varepsilon_{4a}(u) + \varepsilon_{4b}(u) \not \equiv 0 \mod 2$ by \cite[Theorem 4.1]{CohnLivingstone} these equations prove the claim. The fact that $u$ is rationally conjugate to its inverse follows from the character values of $\chi.$

So let $C_4 \times C_2 \cong U = \langle t \rangle \times \langle s \rangle.$ Then by Lemma \ref{ModHeLP}
\[\frac{1}{8}\sum_{u \in U} \chi(u) = \frac{1}{8}(1 + 3 + 2\chi(t) + 2\chi(st))\]
is a non-negative integer. Since $\chi(t), \chi(st) \in \{\pm 1\}$ this implies $\chi(t) = \chi(st)$ and thus $t$ and $st$ are rationally conjugate. Now 
\[\frac{1}{8} \sum_{u \in U} \eta(u) = \frac{1}{8}(4m +3\cdot 0 + 4\eta(t))\]
is also a non-negative integer. Since $\eta$ vanishes on the conjugacy classes $2a$ and $4b$, we obtain that $\eta(t) = \eta(4a)\varepsilon_{4a}(t) = -2m\varepsilon_{4a}(t)$. So $4\eta(t) \equiv 0 \mod 8$ meaning that $\frac{m}{2}$ has to be an integer. This contradicts the existence of $U$ and finalizes the proof of the Theorem. \hfill $\qed$\\

{\textbf{Acknowledgement:} I am thankful to Markus Stroppel for many interesting conversations about actions of classical groups.

\bibliographystyle{amsalpha}
\bibliography{SIPC2C4.bib}

\newcommand{\etalchar}[1]{$^{#1}$}
\providecommand{\bysame}{\leavevmode\hbox to3em{\hrulefill}\thinspace}
\providecommand{\MR}{\relax\ifhmode\unskip\space\fi MR }
\providecommand{\MRhref}[2]{%
  \href{http://www.ams.org/mathscinet-getitem?mr=#1}{#2}
}
\providecommand{\href}[2]{#2}
\begin{thebibliography}{MRSW87}

\bibitem[ABG70]{AlperinBrauerGorenstein}
J.~L. Alperin, R.~Brauer, and D.~Gorenstein, \emph{Finite groups with
  quasi-dihedral and wreathed {S}ylow {$2$}-subgroups.}, Trans. Amer. Math.
  Soc. \textbf{151} (1970), 1--261.

\bibitem[Alp86]{AlperinLocal}
J.~L. Alperin, \emph{Local representation theory}, Cambridge Studies in
  Advanced Mathematics, vol.~11, Cambridge University Press, Cambridge, 1986,
  Modular representations as an introduction to the local representation theory
  of finite groups.

\bibitem[Ari07]{Ari}
\emph{Mini-{W}orkshop: {A}rithmetik von {G}ruppenringen}, Oberwolfach Rep.
  \textbf{4} (2007), no.~4, 3209--3239, Abstracts from the mini-workshop held
  November 25--December 1, 2007, Organized by Eric Jespers, Zbigniew Marciniak,
  Gabriele Nebe and Wolfgang Kimmerle, Oberwolfach Reports. Vol. 4, no. 4.

\bibitem[Ben81]{Bender}
H.~Bender, \emph{Finite groups with dihedral {S}ylow {$2$}-subgroups}, J.
  Algebra \textbf{70} (1981), no.~1, 216--228.

\bibitem[BK07]{KonovalovM11}
V.~Bovdi and A.~Konovalov, \emph{Integral group ring of the first {M}athieu
  simple group}, Groups {S}t. {A}ndrews 2005. {V}ol. 1, London Math. Soc.
  Lecture Note Ser., vol. 339, Cambridge Univ. Press, Cambridge, 2007,
  pp.~237--245.

\bibitem[BK11]{AndreasKimmi}
A.~B{\"a}chle and W.~Kimmerle, \emph{On torsion subgroups in integral group
  rings of finite groups}, J. Algebra \textbf{326} (2011), 34--46.

\bibitem[BM15a]{PSL2p3}
A.~B{\"a}chle and L.~Margolis, \emph{Torsion subgroups in the units of the
  integral group ring of {$\operatorname{PSL}(2,p^3)$}}, Arch. Math. (Basel)
  \textbf{105} (2015), no.~1, 1--11.

\bibitem[BM15b]{HeLPPaper}
A.~B{\"a}chle and L.~Margolis, \emph{Help -- a {G}{A}{P}-package for torsion
  units in integral group rings}, http://arxiv.org/abs/1507.08174, preprint
  (2015).

\bibitem[BS59]{BrauerSuzuki}
Richard Brauer and Michio Suzuki, \emph{On finite groups of even order whose
  {$2$}-{S}ylow group is a quaternion group}, Proc. Nat. Acad. Sci. U.S.A.
  \textbf{45} (1959), 1757--1759.

\bibitem[CCN{\etalchar{+}}85]{Atlas}
J.~H. Conway, R.~T. Curtis, S.~P. Norton, R.~A. Parker, and R.~A. Wilson,
  \emph{Atlas of finite groups}, Oxford University Press, Eynsham, 1985,
  Maximal subgroups and ordinary characters for simple groups, With
  computational assistance from J. G. Thackray.

\bibitem[CL65]{CohnLivingstone}
J.~A. Cohn and D.~Livingstone, \emph{On the structure of group algebras {I}},
  Canadian Journal of Mathematics \textbf{17} (1965), 583--593.

\bibitem[CR90]{CR1}
C.~W. Curtis and I.~Reiner, \emph{Methods of representation theory. {V}ol.
  {I}}, Wiley Classics Library, John Wiley \& Sons, Inc., New York, 1990, With
  applications to finite groups and orders, Reprint of the 1981 original, A
  Wiley-Interscience Publication.

\bibitem[DJ96]{DokuchaevJuriaans}
M.~A. Dokuchaev and S.~O. Juriaans, \emph{Finite subgroups in integral group
  rings}, Canad. J. Math. \textbf{48} (1996), no.~6, 1170--1179.

\bibitem[Gor69]{GorensteinCentralizers}
D.~Gorenstein, \emph{Finite groups the centralizers of whose involutions have
  normal {$2$}-complements}, Canad. J. Math. \textbf{21} (1969), 335--357.

\bibitem[Gow76]{Gow}
R.~Gow, \emph{Schur indices of some groups of {L}ie type}, J. Algebra
  \textbf{42} (1976), no.~1, 102--120.

\bibitem[GW65]{GorensteinWalter}
D.~Gorenstein and J.~H. Walter, \emph{The characterization of finite groups
  with dihedral {S}ylow {$2$}-subgroups. {I}}, J. Algebra \textbf{2} (1965),
  85--151.

\bibitem[HB82]{HuppertIII}
Bertram Huppert and Norman Blackburn, \emph{Finite groups. {III}}, Grundlehren
  der Mathematischen Wissenschaften [Fundamental Principles of Mathematical
  Sciences], vol. 243, Springer-Verlag, Berlin-New York, 1982.

\bibitem[Her01]{HertweckIso}
M.~Hertweck, \emph{A counterexample to the isomorphism problem for integral
  group rings}, Ann. of Math. (2) \textbf{154} (2001), no.~1, 115--138.

\bibitem[Her07]{HertweckBrauer}
M.~Hertweck, \emph{Partial {A}ugmentations and {B}rauer character values of
  torsion units in group rings}, http://arxiv.org/abs/math/0612429, preprint
  (2007).

\bibitem[Her08]{HertweckCpCp}
M.~Hertweck, \emph{Unit groups of integral finite group rings with no noncyclic
  abelian finite {$p$}-subgroups}, Comm. Algebra \textbf{36} (2008), no.~9,
  3224--3229.

\bibitem[HHK09]{HertweckHoefertKimmerle}
M.~Hertweck, C.R. H{\"o}fert, and W.~Kimmerle, \emph{Finite groups of units and
  their composition factors in the integral group rings of the group {${\rm
  PSL}(2,q)$}}, J. Group Theory \textbf{12} (2009), no.~6, 873--882.

\bibitem[Hig40]{HigmanThesis}
G.~Higman, \emph{Units in group rings}, D. phil. thesis, Oxford Univ., 1940.

\bibitem[H{\"o}f08]{Hoefert}
C.R. H{\"o}fert, \emph{Bestimmung von {K}ompositionsfaktoren endlicher
  {G}ruppen aus {B}urnsideringen und ganzzahligen {G}ruppenringen},
  Doktorarbeit, Universit{\"a}t Stuttgart, 2008,
  http://elib.uni-stuttgart.de/opus/frontdoor.php?source\_opus=3541\&la=de.

\bibitem[JdR16]{GRG1}
E.~Jespers and {\'A}.~del R{\'{\i}}o, \emph{{Group ring groups. Volume 1:
  Orders and generic constructions of units.}}, Berlin: De Gruyter, 2016.

\bibitem[Kim07]{KimmiC2C2}
W.~Kimmerle, \emph{Torsion units in integral group rings of finite insoluble
  groups}, Oberwolfach Reports \textbf{4} (2007), no.~4, 3229--3230, Abstracts
  from the mini-workshop held November 25--December 1, 2007, Organized by Eric
  Jespers, Zbigniew Marciniak, Gabriele Nebe, and Wolfgang Kimmerle.

\bibitem[Kim15]{KimmerleSylow}
\bysame, \emph{Sylow like theorems for $\mathrm{V}(\mathbb{Z}{G})$}, Int. J.
  Group Theory \textbf{4} (2015), no.~4, 49--59.

\bibitem[KS04]{Kurzweil}
H.~Kurzweil and B.~Stellmacher, \emph{The theory of finite groups. an
  introduction}, Universitext, Springer-Verlag, New York, 2004, Translated from
  the 1998 German original.

\bibitem[LP89]{LutharPassiA5}
I.S. Luthar and I.B.S. Passi, \emph{Zassenhaus conjecture for {$A_5$}}, Proc.
  Indian Acad. Sci. Math. Sci. \textbf{99} (1989), no.~1, 1--5.

\bibitem[MRSW87]{MarciniakRitterSehgalWeiss}
Z.~Marciniak, J.~Ritter, S.~K. Sehgal, and A.~Weiss, \emph{Torsion units in
  integral group rings of some metabelian groups {II}}, J. Number Theory
  \textbf{25} (1987), no.~3, 340--352.

\bibitem[San81]{Sandling}
R.~Sandling, \emph{Graham {H}igman's thesis ``{U}nits in group rings''},
  Integral representations and applications ({O}berwolfach, 1980), Lecture
  Notes in Math., vol. 882, Springer, Berlin-New York, 1981, pp.~93--116.

\bibitem[Val94]{Valenti}
Angela Valenti, \emph{Torsion units in integral group rings}, Proc. Amer. Math.
  Soc. \textbf{120} (1994), no.~1, 1--4.

\bibitem[Wei88]{Weiss88}
A.~Weiss, \emph{Rigidity of {$p$}-adic {$p$}-torsion}, Ann. of Math. (2)
  \textbf{127} (1988), no.~2, 317--332.

\bibitem[Whi13]{White}
D.~L. White, \emph{Character degrees of extensions of {${\rm PSL}_2(q)$} and
  {${\rm SL}_2(q)$}}, J. Group Theory \textbf{16} (2013), no.~1, 1--33.

\bibitem[Zas35]{ZassenhausKennzeichnung}
H.~Zassenhaus, \emph{Kennzeichnung endlicher linearer {G}ruppen als
  {P}ermutationsgruppen}, Abh. Math. Sem. Univ. Hamburg \textbf{11} (1935),
  no.~1, 17--40.

\end{thebibliography}

\bigskip

Leo Margolis, Departamento de Matem\'aticas, Facultad de Matem\'aticas, Universidad de Murcia, 30100 Murcia, Spain.
\emph{leo.margolis@um.es}

\end{document}